\documentclass[11pt]{amsart}
\usepackage{amsfonts}
\usepackage{ifthen}
\usepackage{amsthm}
\usepackage{amsmath}
\usepackage{graphicx}
\usepackage{amscd,amssymb,amsthm}
\usepackage{graphicx}
\usepackage{hyperref}

\newcounter{minutes}
\divide\time by 60
\newcounter{hours}
\multiply\time by 60 \addtocounter{minutes}{-\time}

\title[Zeros of a cross-product]{Zeros of a cross-product of the Coulomb wave and Tricomi hypergeometric functions}

\author[\'A. Baricz]{\'Arp\'ad Baricz}
\address{Institute of Applied Mathematics, \'Obuda University, Budapest, Hungary}
\address{Department of Economics, Babe\c{s}-Bolyai University, Cluj-Napoca, Romania}
\email{bariczocsi@yahoo.com}

\thanks{$^{\bigstar}$The research of \'A. Baricz was supported by the J\'anos Bolyai Research Scholarship of the Hungarian Academy of Sciences. The author is very grateful to Prof. Mourad E.H. Ismail for suggesting the investigation of the zeros discussed in this paper of the cross-product of the regular Coulomb wave function and Tricomi hypergeometric function, and also for his kind hospitality during the author's visit at Department of Mathematics of the City University of Hong Kong in September 2011. The author is also very grateful to his friend and colleague Prof. Alexandru Krist\'aly for the discussions and suggestions about the boundary value problems considered in this paper.}

\newtheorem{theorem}{Theorem}

\keywords{Coulomb wave function; Tricomi hypergeometric function; boundary value problem; zeros of a cross-product; eigenvalues; eigenfunctions; Bessel and modified Bessel functions; monotonicity of the zeros.} \subjclass[2010]{34B09, 34B30, 33C15, 33C10.}
\begin{document}

\def\thefootnote{}
\footnotetext{ \texttt{File:~\jobname .tex,
          printed: \number\year-0\number\month-\number\day,
          \thehours.\ifnum\theminutes<10{0}\fi\theminutes}
} \makeatletter\def\thefootnote{\@arabic\c@footnote}\makeatother

\maketitle

\begin{center}
Dedicated to Prof. P\'eter T. Nagy on the occasion of his
70th birthday
\end{center}

\begin{abstract}
Motivated by a problem related to conditions for the existence of clines in genetics,
in this note our aim is to show that the positive zeros of a cross-product
of the regular Coulomb wave function and the Tricomi hypergeometric function
are increasing with respect to the order. In particular, this implies that
the eigenvalues of a boundary value problem are increasing with the dimension.
\end{abstract}


\section{Introduction}

In his study about the existence of clines in genetics, Nagylaki \cite{nagy} considered
a partial differential equation in space and time satisfied by the gene frequency in a
monoecious population distributed continuously over an arbitrary habitat. Nagylaki \cite{nagy}
showed that this partial differential equation reduces to the simplest multidimensional
generalization of the classical Fisher-Haldane
cline model, and he investigated the efficacy of migration and selection in maintaining genetic
variability at equilibrium in this model by deducing conditions
for the existence of clines under various circumstances. The boundary value problem
considered by Nagylaki reads as follows
\begin{equation}\label{eqnagy}
\Delta p+\lambda^2g(r)p=0,
\end{equation}
where $p'(0)=0$ and $p(\infty)<\infty,$ $\Delta p$ is the $n$-dimensional Laplacian,
$r$ is the distance from the origin of an $n$-dimensional vector $x,$ and
$$g(r)=\left\{\begin{array}{ll}1,& r\in[0,1]\\-\alpha^2,& r>1\end{array}\right. .$$
Nagylaki \cite{nagy} conjectured that for each $\alpha>0$ fixed the smallest positive eigenvalues
of the above boundary value problem increase with the dimension. Motivated by Nagylaki's investigation,
Ismail and Muldoon \cite{ismail} considered the radial part of the boundary value problem \eqref{eqnagy}, that is,
\begin{equation}\label{eqism}
-(ry'(r))'+\nu^2r^{-1}y(r)=\lambda^2rg(r)y(r),
\end{equation}
where $p'(0)=0,$ $p(\infty)<\infty,$ $y(r)=r^{\nu}p(r)$ and $\nu={n}/{2}-1,$ and they showed that the positive eigenvalues of \eqref{eqism}
are the positive zeros of fixed rank of the cross-product $$J_{\nu+1}(r)K_{\nu}(\alpha r)-\alpha K_{\nu+1}(\alpha r)J_{\nu}(r),$$ where $J_{\nu}$
is the Bessel function of the first kind, while $K_{\nu}$ stands for the modified Bessel function of the second kind. Moreover, motivated by Askey's claim, Ismail and Muldoon \cite{ismail} proved that the positive zeros of the cross-product $$J_{\nu+\beta}(r)K_{\nu}(\alpha r)-\alpha^{\beta} K_{\nu+\beta}(\alpha r)J_{\nu}(r)$$
are increasing with respect to $\nu$ on $[-\beta/2,\infty),$ where $\beta\in(0,1].$ Thus, it is clear that Nagylaki's conjecture follows from the case $\beta=1$ of the above result. In \cite{ismail} the authors actually stated more: they showed that the expression $\alpha r$ in the above affirmation can be changed to any strictly increasing differentiable function on $(0,\infty)$ and instead of $\alpha^{\beta}$ it can be taken an arbitrary positive constant. Motivated by the importance of the boundary value problem \eqref{eqnagy} and its radial part \eqref{eqism} in the existence of clines, and by following the suggestion of Ismail, in this note our aim is to show that Nagylaki's claim on the positive eigenvalues will be also true if we consider a more general setting, that is, if we change the Bessel function of the first kind to the regular Coulomb wave function, and the modified Bessel function of the second kind to Tricomi hypergeometric function of the second kind. This is actually a generalization of the problem considered by Nagylaki. For more details on the special functions appearing in this paper we refer to \cite{nist}.

\section{The eigenvalue problem related to Coulomb and Tricomi functions}
\setcounter{equation}{0}

In order to extend Nagylaki's problem we consider the next boundary value problem
\begin{equation}\label{eqnew}
\Delta p+\varphi_{\lambda}(r)r^{-2}p=0,
\end{equation}
where $p'(0)=0$ and $p(\infty)<\infty,$
$$\varphi_{\lambda}(r)=L(L-1)+\lambda^2r^2g(r)-2\eta\lambda rh(r),$$
$L=(n-1)/2,$ $\eta$ is a real parameter, and
$$h(r)=\left\{\begin{array}{ll}1,& r\in[0,1]\\\alpha,& r>1\end{array}\right. .$$
It can be shown that the radial part of the above boundary value problem \eqref{eqnew} is
\begin{equation}\label{radnew}
r^2y''(r)+\left(\lambda^2r^2g(r)-2\eta\lambda rh(r)-L(L+1)\right)y(r)=0,
\end{equation}
where $p'(0)=0,$ $p(\infty)<\infty$ and $y(r)=r^{L}p(r).$ Now, if we suppose that $r\in(0,1],$ then we arrive at
$$r^2y''(r)+\left(\lambda^2r^2-2\eta\lambda r-L(L+1)\right)y(r)=0.$$ By using the change of variable $u=\lambda r$ (and taking $y(r)=z(u)$), the above equation becomes the Coulomb wave equation
$$u^2z''(u)+\left(u^2-2\eta u-L(L+1)\right)z(u)=0.$$
Moreover, when $r>1$ the equation \eqref{radnew} becomes
$$r^2y''(r)-\left(\alpha^2\lambda^2r^2+2\eta\alpha\lambda r-L(L+1)\right)y(r)=0,$$
which after the change of variable $v=\alpha\lambda r$ (and taking $y(r)=q(v)$) becomes a transformation of the Kummer confluent hypergeometric
differential equation
$$v^2q''(v)-\left(v^2+2\eta v+L(L+1)\right)q(v)=0.$$
Thus, when $r\in(0,1]$ the differential equation \eqref{radnew} has as particular solution the regular Coulomb wave function $$y(r)=A\cdot F_{L}(\eta,u),$$ where $A$ is a real constant, while for $r>1$ the equation \eqref{radnew} has the particular solution a transformation of the Tricomi hypergeometric function $$y(r)=B\cdot v^{L+1}e^{-v}\psi(L+\eta+1,2L+2,2v),$$ where $B$ is a real constant. It is important to mention here that when $\eta=0$ the above particular solutions reduce to $$y(r)=A\cdot \sqrt{\frac{\pi}{2u}}J_{L+\frac{1}{2}}(u)\ \ \mbox{and}\ \  y(r)=B\cdot 2^{-L}\sqrt{\frac{2v}{\pi}}K_{L+\frac{1}{2}}(v),$$ which show that the boundary value problem \eqref{eqnew} is a natural extension of \eqref{eqnagy}, while \eqref{radnew} is a natural extension of \eqref{eqism}.

Now, we are ready to state the main result of this paper.

\begin{theorem}\label{the1}
The following assertions are valid:
\begin{enumerate}

\item[\bf a.] The boundary value problem \eqref{radnew} has for its eigenvalues the zeros of the cross-product of regular Coulomb wave and Tricomi hypergeometric functions $$F_{L}'(\eta,r)Q_{L}(\eta,\alpha r)-\alpha Q_{L}'(\eta,\alpha r)F_{L}(\eta,r)$$ and corresponding eigenfunctions
$$r\mapsto \Theta_{L}(\eta,r)=\left\{\begin{array}{ll}Q_{L}(\eta,\alpha\lambda)\cdot F_{L}(\eta,\lambda r),& r\in(0,1]\\
F_{L}(\eta,\lambda)\cdot Q_L(\eta,\alpha\lambda r),& r>1\end{array}\right. ,$$
where $$Q_{L}(\eta,r)=r^{L+1}e^{-r}\psi(L+\eta+1,2L+2,2r).$$

\item[\bf b.] For fixed $\alpha>0,$ $\eta\in\mathbb{R}$ such that $L+\eta>0,$ and $L>-3/2,$ $L\neq-1$ if $\eta\neq0$ and $L>-3/2$ if $\eta=0,$ the equation
\begin{equation}\label{eqtrans}
F_{L}'(\eta,r)/F_{L}(\eta,r)=\alpha Q_{L}'(\eta,\alpha r)/Q_{L}(\eta,\alpha r)
\end{equation}
has infinitely many positive roots, which we denote in increasing order by $\lambda_{L,\eta,\alpha,n},$ $n\in\mathbb{N}.$ These zeros satisfy $$x_{L,\eta,n-1}<\lambda_{L,\eta,\alpha,n}<x_{L,\eta,n},$$ $n\in\{2,3,\dots\},$ where $x_{L,\eta,n}$ stands for the
$n$th positive zero of the Coulomb wave function $\rho\mapsto F_{L}(\eta,\rho).$ Moreover, if $\alpha>0,$ $\eta\in\mathbb{R}$ and $L>-1/2,$ then we have $\lambda_{L,\eta,\alpha,1}<x_{L,\eta,1}.$

\item[\bf c.] For fixed $\alpha>0,$ $\eta\geq0$ and $n\in\mathbb{N}$ the zeros $\lambda_{L,\eta,\alpha,n}$ increase with $L$ on $[0,\infty).$

\end{enumerate}
\end{theorem}

We note that since the boundary value problem \eqref{eqnew} is an extension of \eqref{eqnagy}, while \eqref{radnew} is an extension of \eqref{eqism},
if we take $\eta=0$ in the above theorem, then we obtain some of the main results from \cite{ismail} for the case $\beta=1.$ In the proof of our main result we followed the approach considered in \cite{ismail}, namely, the Sturmian-type arguments and the approach of the Hellman-Feynman theorem of quantum chemistry. Moreover, we used some recent results on regular Coulomb wave and Tricomi hypergeometric functions: the so-called Mittag-Leffler expansion of regular Coulomb wave function (obtained from the infinite product representation, see \cite{coulomb,stampach,wimp} for more details), and a Tur\'an type inequality for Tricomi hypergeometric functions, recently obtained and written in terms of a monotonicity result (see \cite[Remark 3]{baricz}).

\begin{proof}[Proof of Theorem \ref{the1}]
{\bf a.} Subject to the stated boundary condition the differential equation in \eqref{radnew} has solution
$$y(r)=\left\{\begin{array}{ll}A\cdot F_{L}(\eta,\lambda r),& r\in(0,1]\\B\cdot Q_L(\eta,\alpha\lambda r),& r>1\end{array}\right..$$
Since $y$ and $y'$ are to be continuous at $r=1$ we must have $$A\cdot F_{L}(\eta,\lambda)=B\cdot Q_L(\eta,\alpha\lambda)$$ and $$A\cdot \lambda F_{L}'(\eta,\lambda)=B\cdot \alpha\lambda Q_L'(\eta,\alpha\lambda),$$ and there will be a notrivial solution of this system if and only if $$F_{L}'(\eta,\lambda)Q_{L}(\eta,\alpha \lambda)=\alpha Q_{L}'(\eta,\alpha\lambda)F_{L}(\eta,\lambda).$$ In this case we may take $A=Q_L(\eta,\alpha\lambda)$ and $B=F_L(\eta,\lambda).$ Thus, indeed the boundary value problem \eqref{radnew} has for its eigenvalues the zeros of the cross-product of the regular Coulomb wave and Tricomi hypergeometric functions, that is, $$F_{L}'(\eta,r)Q_{L}(\eta,\alpha r)-\alpha Q_{L}'(\eta,\alpha r)F_{L}(\eta,r)$$ and corresponding eigenfunctions
$r\mapsto\Theta_{L}(\eta,r).$

{\bf b.} The equation \eqref{eqtrans} is equivalent to
\begin{equation}\label{eqtrans2}\frac{F_{L}'(\eta,r)}{F_{L}(\eta,r)}-\frac{L+1}{r}=\alpha+\frac{2\alpha\,\psi'(L+\eta+1,2l+2,2\alpha r)}{\psi(L+\eta+1,2L+2,2\alpha r)}.\end{equation}
Since for $L>-3/2,$ $L\neq-1$ if $\eta\neq0$ and $L>-3/2$ if $\eta=0$ we have (see \cite[Lemma 1]{coulomb} for more details) $$\frac{F_{L}'(\eta,r)}{F_{L}(\eta,r)}-\frac{L+1}{r}=\frac{\eta}{L+1}-\sum_{n\geq1}\left(\frac{r}{x_{L,\eta,n}(x_{L,\eta,n}-r)}+
\frac{r}{y_{L,\eta,n}(y_{L,\eta,n}-r)}\right),$$
the left-hand side of the equation \eqref{eqtrans2} is decreasing on $(0,x_{L,\eta,1})$ and also on each interval $(x_{L,\eta,n},x_{L,\eta,n+1}),$ $n\in\mathbb{N}.$ Here $y_{L,\eta,n}$ stands for the $n$th negative zero of the regular Coulomb wave function $r\mapsto F_{L}(\eta,r).$ When $r\searrow0$ the left-hand side of \eqref{eqtrans2} tends to $\eta/(L+1),$ when $r\nearrow x_{L,\eta,n},$ $n\in\mathbb{N}$ it tends
to $-\infty$ and when $r\searrow x_{L,\eta,n},$ $n\in\mathbb{N}$ it tends
to $+\infty.$ On the other hand, according to \cite[Remark 3]{baricz} we know that for $a>1$ and $c\in\mathbb{R}$ the function $r\mapsto \psi'(a,c,r)/\psi(a,c,r)$ is increasing on $(0,\infty).$ Moreover, by using the recurrence relation
$$\psi'(a,c,r)=-a\psi(a+1,c+1,r),$$
the fact that $\psi(a,c,r)$ is positive for $a,c,r\in\mathbb{R},$ and the asymptotic expansion
$$\psi(a,c,r)=r^{-a}(1+\mathcal{O}(r^{-1}))\ \ \mbox{as}\ \  r\to\infty,$$ it follows that $r\mapsto \psi'(a,c,r)/\psi(a,c,r)$ maps $(0,\infty)$ into
$(-\infty,0).$ Thus, the right-hand side of \eqref{eqtrans2} is increasing on $(0,\infty)$ for $\alpha>0$ and $L+\eta>0$ and maps the interval $(0,\infty)$ into $(-\infty,\alpha).$ These show that the equation \eqref{eqtrans2} has indeed infinitely many positive roots, and starting from the second positive root they are certainly located between the positive zeros of the regular Coulomb wave function. Now, by using the asymptotic relation $$\psi(a,c,r)\sim\Gamma(c-1)r^{1-c}/\Gamma(a)\ \ \mbox{as}\ \ r\to0\ \ \mbox{and}\ \ \ c>1,$$ it follows that $\psi'(a,c,r)/\psi(a,c,r)\sim(1-c)/r$ as $r\to0$ and $c>1,$ and thus the right-hand side of \eqref{eqtrans2} tends to $-\infty$ as $r\to0$ and $L>-1/2.$ This shows that indeed if $\alpha>0,$ $\eta\in\mathbb{R}$ and $L>-1/2,$ then we have $\lambda_{L,\eta,\alpha,1}<x_{L,\eta,1}.$

{\bf c.} Since $r\mapsto\Theta_L(\eta,r)$ are eigenfunctions of the boundary value problem \eqref{radnew}, we have
$$-\Theta_L''(\eta,r)\Theta_L(\eta,r)+2\eta\lambda \frac{1}{r}h(r)\Theta_L^2(\eta,r)+L(L+1)\frac{1}{r^2}\Theta_L^2(\eta,r)=\lambda^2g(r)\Theta_L^2(\eta,r).$$
Integrating from zero to infinity we get
\begin{align*}\int_0^{\infty}&\left(\lambda^2g(r)-2\eta\lambda \frac{1}{r}h(r)\right)\Theta_L^2(\eta,r)dr\\&=L(L+1)\int_0^{\infty}\frac{1}{r^2}\Theta_L^2(\eta,r)dr+
\int_0^{\infty}\left(\Theta_L'(\eta,r)\right)^2dr,\end{align*}
where we used integration by parts in the last integral. Since the right-hand side of the above relation is positive for $L\geq0,$ it follows that
\begin{equation}\label{ineqlambda}\lambda^2\int_0^{\infty}g(r)\Theta_L^2(\eta,r)dr\geq \lambda\int_0^{\infty}2\eta\frac{1}{r}h(r)\Theta_L^2(\eta,r)dr.\end{equation}
Now, writing $\lambda_L$ instead of $\lambda_{L,\eta,\alpha,n},$ multiplying the equations
$$-\Theta_L''(\eta,r)+2\eta\lambda_L \frac{1}{r}h(r)\Theta_L(\eta,r)+L(L+1)\frac{1}{r^2}\Theta_L(\eta,r)=\lambda_{L}^2g(r)\Theta_L(\eta,r),$$
$$-\Theta_M''(\eta,r)+2\eta\lambda_M \frac{1}{r}h(r)\Theta_M(\eta,r)+M(M+1)\frac{1}{r^2}\Theta_M(\eta,r)=\lambda_{M}^2g(r)\Theta_M(\eta,r)$$
by $\Theta_M(\eta,r),$ $\Theta_L(\eta,r)$ respectively, subtracting and integrating between $0$ and $\infty$ we get
\begin{align*}&\left(\Theta_M'(\eta,r)\Theta_L(\eta,r)-\Theta_L'(\eta,r)\Theta_M(\eta,r)\right)\Big|_0^{\infty}\\
&+(\lambda_L-\lambda_M)\int_0^{\infty}2\eta\frac{1}{r}h(r)\Theta_L(\eta,r)\Theta_M(\eta,r)dr\\
&+\left[L(L+1)-M(M+1))\right]\int_0^{\infty}\frac{1}{r^2}\Theta_L(\eta,r)\Theta_M(\eta,r)dr\\
&=(\lambda_L^2-\lambda_M^2)\int_0^{\infty}g(r)\Theta_L(\eta,r)\Theta_M(\eta,r)dr.\end{align*}
Note that the integrated term vanishes at $0$ and $\infty$ for $L,M>0,$ and consequently dividing both parts of the above equation by
$L-M$ and taking the limit $M\to L$ we obtain
$$\frac{d\lambda_L}{dL}\int_0^{\infty}2\eta\frac{1}{r}h(r)\Theta_L^2(\eta,r)dr+(2L+1)\int_0^{\infty}\frac{1}{r^2}\Theta_L^2(\eta,r)dr$$
$$=\frac{d\lambda_L^2}{dL}\int_0^{\infty}g(r)\Theta_L^2(\eta,r)dr$$
or equivalently
$$\frac{d\lambda_L}{dL}\int_0^{\infty}2\eta\frac{1}{r}h(r)\Theta_L^2(\eta,r)dr+(2L+1)\int_0^{\infty}\frac{1}{r^2}\Theta_L^2(\eta,r)dr$$
$$=2\lambda_L\frac{d\lambda_L}{dL}\int_0^{\infty}g(r)\Theta_L^2(\eta,r)dr.$$
This implies that
$$\frac{d\lambda_L}{dL}\left(2\lambda_L\int_0^{\infty}g(r)\Theta_L^2(\eta,r)dr-\int_0^{\infty}2\eta\frac{1}{r}h(r)\Theta_L^2(\eta,r)dr\right)$$
$$=(2L+1)\int_0^{\infty}\frac{1}{r^2}\Theta_L^2(\eta,r)dr>0,$$
which in view of \eqref{ineqlambda} and the fact that $\lambda_L$ is positive according to part {\bf b}, yields that $d\lambda_L/dL>0,$ that is, indeed for fixed $\alpha>0$ and $\eta\geq0$ the zero $\lambda_{L}$ is increasing with respect to $L$ on $[0,\infty).$
\end{proof}

\end{document}